\documentclass[11pt]{amsart}
\usepackage{amsmath, amssymb}
\usepackage{amsfonts}
\usepackage{amsthm}
\usepackage{amsopn}
\usepackage{amscd}

\usepackage [all]{xy}

\input xy

\xyoption{arc}

\newcommand{\bR}{\mathbf{R}}
\newcommand{\bC}{\mathbf{C}}

\newcommand{\tensor}{\otimes}
\newcommand{\comp}{\circ}

\newcommand{\vect}{\mathbf{Vect}}
\newcommand{\sB}{\mathcal{B}}
\newcommand{\dif}{\text{Diff}}

\theoremstyle{plain}
\newtheorem{thm}{Theorem}[section]

\newtheorem{lem}[thm]{Lemma}

\theoremstyle{definition}
\newtheorem{defn}[thm]{Definition}

\theoremstyle{remark}

\let\g\gamma

\let\e\varepsilon

\let\l\lambda
\let\m\mu
\let\n\nu

\let\f\varphi

\let\G\Gamma

\let\Si\Sigma

\let\ra\rightarrow
\let\lra\longrightarrow
\let\ti\tilde

\begin{document}
\title{On 2-dimensional topological field theories}
\date{\today}
\author{Florin Dumitrescu}
\maketitle

\begin{abstract}
In this paper we give a characterization of 2-dimensional topological field theories over a space $X$ as Frobenius bundles with connections over $LX$, the free loop space of $X$. This is a generalization of the folk theorem stating that 2-dimensional topological field theories (over a point) are described by finite-dimensional commutative Frobenius algebras. In another direction, this result extends the description of 1-dimensional topological field theories over a space $X$ as vector bundles with connections over $X$, cf. \cite{DST}.
\end{abstract}

\vspace{.1in} 

In \cite{A}, Atiyah introduces the notion of a $d$-dimensional topological quantum field theory. About the same time, Segal \cite{Se1} defines the concept of a 2-dimensional conformal field theory, motivated by the problem of avoiding the difficulties of Feynman path-integration in quantum field theory through an axiomatic approach. In \cite{Se2} he suggests that 2-dimensional conformal field
theories  based on a manifold $X$ should provide geometric 
cocycles for some version of elliptic cohomology of $X$, the home of elliptic genera such as the Witten genus (see \cite{W1}). The idea of relating
field theories and cohomology theories was elaborated
by Stolz-Teichner \cite{ST}, and their collaboration confirms this
relationship  in dimension one: the space of 1-dimensional supersymmetric {\it euclidean}
field theories is a classifying space for K-theory (supersymmetry here avoids some topological triviality).

Understanding field theories of various flavors seems like a very interesting problem lying at the intersection of topology, geometry and quantum field theory. One of the main conjectures by Stolz-Teichner states that the space of 2-dimensional supersymmetric {\it euclidean} field theories is a classifying space for the theory of topological modular forms (see \cite{ST} and \cite{ST2}). In this paper we deal with the simpler case of 2-dimensional  {\it topological } field theories over a space $X$, of which more is known. A folk theorem states that 2-dimensional topological field theories (over $X=\star$) are given by finite-dimensional commutative Frobenius algebras (see for example \cite{Abrams},\cite{Kock} or \cite{MS}). Topological 1-dimensional field theories over a space $X$ are given by finite-dimensional vector bundles with connections over $X$, see \cite{DST}. In this paper we show that 2-dimensional  topological field theories over $X$ are given by Frobenius bundles with connections over $LX$, the free loop space of $X$. A Frobenius bundle over $LX$ is a vector bundle over $LX$ whose restriction to $X$, the space of constant loops in $LX$, is a bundle of Frobenius algebras, and the fiber over an arbitrary loop $\g$ in $X$, based at a point $x\in X$, admits an action of the Frobenius algebra which is the fiber over the constant loop at $x$.

A topological field theory is a functor from a bordism category to an algebraic category, usually the category of topological vector spaces. There are various versions of such topological theories, for example one could modify the bordism category and consider {\it open}, or {\it open-and-closed} bordisms. Two-dimensional such theories were characterized by Moore-Segal \cite{MS} and Lauda-Pfeiffer \cite{LP}. One could also replace the target category by the category of complexes (see Costello \cite{Co1}) or by an arbitrary symmetric monoidal category. Even further, one could replace categories by higher categories and consider {\it extended} topological field theories. Such theories were characterized in the two-dimensional case by Schommer-Pries \cite{Schommer} and in general by Lurie, who outlines in  \cite{Lu}  the proof of the cobordism hypothesis, a conjecture stated by Baez-Dolan \cite{BD}. None of these variations on topological field theories will be considered in this paper.

This paper is organized as follows: in section 1 we define the notion of a 2-dimensional topological field theory  (2-TFT, for short) over a space $X$ (definition \ref{2TFT}) in a manner convenient for our purposes, and the notion of a Frobenius bundle with connection over $LX$ (definition \ref{FB}). We then state the main theorem \ref{main} which establishes the equivalence of the two notions. Section 2 is dedicated to the proof of the theorem. Further consequences of our definition of Frobenius bundles such as Frobenius actions and $\dif(S^1)^+$-actions are relegated to section 3. We also talk about holonomy along closed surfaces and rank-one 2-TFTs which are basically $S^1$-bundle gerbes with connections.  \\

\noindent \textbf{Acknowledgements.} This paper was written while the author was visiting Max-Planck Institute for Mathematics in Bonn. We would like to thank the Institute for support and the friendly, stimulating environment provided. We would also like to give special thanks to Ralph Cohen, Pokman Cheung, Stephan Stolz, Peter Teichner and Konrad Waldorf for helpful conversations and interest in the project.

\section{Definitions and statement of results}

\noindent \textbf{Field theories.} A {\it $d$-dimensional field theory over a space $X$} is generally defined to be a smooth functor of symmetric monoidal categories
\[E: \sB^d(X) \ra \vect.\]
The objects in the category $\sB^d(X)$ are pairs $(Y, \g)$ consisting of a closed oriented $d-1$ manifold $Y$ and a smooth map $\g:Y\ra X$. The morphisms between
two such objects $(Y_1, \g_1)$ and $(Y_2, \g_2)$ are pairs $(Z, \Si)$, with $Z$ an oriented $d$-manifold such that $\partial Z= Y_1 \amalg \bar{Y}_2$, where $\bar{Y}$ denotes $Y$ with opposite orientation, and $\partial \Si= \g_1\amalg \g_2$. The category $\vect$ is the category of  topological vector spaces over a field $k$, which is usually taken to be $\bR$ or $\bC$. ``Monoidal" means disjoint union $\amalg$ and tensor product $\tensor$ in the category $\sB^d(X)$ and $\vect$, respectively. The functor $E$ is compatible with these monoidal structures, that is
\[ E(\g_1\amalg \g_2)= E(\g_1)\tensor E(\g_2), \]
and similarly on morphisms. Moreover, $E$ takes the empty set to our ground field $k$. ``Functoriality" means that glueing of bordisms in $\sB^d(X)$ corresponds to composition of linear maps in $\vect$. A functor is \textit{smooth} if it maps a  smooth family (i.e. parametrized by smooth manifolds) of objects in the source category to a  smooth family of objects in the target category, and similarly for morphisms. 

There are various flavors of field theories according to the geometric structures we require on bordisms: \textit{topological} (no structure), \textit{euclidean} (flat Riemannian metric), \textit{conformal} (conformal structure) etc. The easiest example of a topological field theory (TFT, for short) is for $d=1$, in which case it entails to a parametrization-invariant parallel transport associated to a vector bundle over $X$, which in turn is just a vector bundle with connection over $X$ (see \cite{DST}). 

We will modify slightly the definition above to avoid technical difficulties arising from glueing bordisms in $X$ along common boundaries. One way to deal with this is to consider objects along with collars, and glue along collars. Our approach is to replace the composition of bordisms by decomposition. The price to pay is to give up the beautiful categorical language or modify the definition of category accordingly. In order to avoid set-theoretical issues, we should also require that all the vector spaces appearing in the definition below are subspaces of a {\it fixed} infinite-dimensional topological vector space, let us say $k^\infty$. Since we are dealing  with {\it topological} field theories, it turns out that all the vector spaces appearing below are finite dimensional.

%  2-TFTs
\begin{defn} \label{2TFT} A {\it 2-dimensional topological field theory $E$ over $X$} assigns smoothly to a union of loops $\amalg\g:\amalg S^1\ra X$ in $X$ a topological vector space $E(\amalg\g)=\tensor E(\g)$ and to a surface in $X$, i.e. a map $\Si$ from an oriented surface $Z$ in $X$ a continuous linear map $E(\Si): E(\partial \Si_{in})\ra  E(\partial \Si_{out})$, where the boundary $\partial Z= \partial Z_{in}\amalg \partial Z_{out}$ splits into incoming and outgoing boundary according to whether the orientation of the circles coincides with the induced orientation of the surface or its reverse, so that the properties below hold:

\begin{enumerate}

\item (monoidal structure preserving) As noted above, if $\amalg\g$ is a union of loops in $X$, we require
$$E(\amalg\g)=\tensor E(\g).$$
Moreover, we should ask that $E(\emptyset)=\bC$. Also, if $\Si_1$ and $\Si_2$ are two bordisms in $X$, then we should have 
\[ E(\Si_1\amalg \Si_2)= E(\Si_1)\tensor E(\Si_2). \]

\item (compatibility under decomposition)  If $\Si:Z\ra X$ is decomposed along a (union of) circle(s) $Y_0$ in $Z$ so that $Z=Z_1\amalg_{Y_0} Z_2$, and $\Si_i:=\Si|_{Z_i},\ i=1,2$, we have
\[E(\Si)= E(\Si_2)\comp E(\Si_1)\]
where the left-hand side is a map $E(\partial\Si_{in}) \ra E(\partial\Si_{out})$, and the right-hand side is a composition 
$E(\partial\Si_{in})\ra E(\Si_{|Y_0})\ra E(\partial\Si_{out})$.

\item (invariance under diffeomorphisms)
\[ E(\Si\comp\Phi)= E(\Si), \]
for $\Phi$ an arbitrary diffeomorphism of surfaces that is the identity on the boundary.

\item (identity preserving) Let $\Si_\g$ be the ``constant" bordism over $\g$, i.e. the cylinder over $\g$ (since our theory is topological, the height of the cylinder is irrelevant here). Then 
\[ E(\Si_\g)= 1_{E(\g)}. \]

\end{enumerate}

\end{defn}

The purpose of this article is to understand 2-dimensional topological field theories over a manifold $X$. Let  $LX$ denote the free loop space of the manifold $X$.  Our main result is the following

% Main theorem
\begin{thm}  \label{main} There is a 1-1 correspondence:

\[  \left\{
\begin{array}{l}
\text{2-dim topological field theories}\\
\text{\ \ \ \ \ \ \ \ \ \ \ \ \ over $X$}
\end{array} \right\}
\leftrightarrow \left\{
\begin{array}{l}
\text{Frobenius bundles with connections}\\
\text{\ \ \ \ \ \ \ \ \ \ \ \ \ \ \ \ over $LX$}
\end{array} \right\} \]
\end{thm}

A Frobenius bundle with connection on a loop space $LX$ encodes some algebraic data (multiplication and comultiplication maps coming from 8-like loops in $X$) and some geometric data (parallel transport along paths in $LX$- or cylinders in $X$) in a compatible manner. More precisely, we have the following

%Frobenius bundles
\begin{defn} \label{FB}
A {\it Frobenius bundle with connection} on $LX$, the free loop space on a manifold $X$, is a vector bundle $A$ over $LX$ 
together with the following data:
\begin{itemize}
\item For a loop $\g$ in $X$, denote by $A_\g$ the fiber of the bundle $A$ at $\g$. If 
$\g=\g_1*\g_2$ is the concatenation of $\g_1$ and $\g_2$, then there are maps:
\[  \m: A_{\g_1}\tensor A_{\g_2} \ra A_\g, \text{ and } \n:A_\g\ra A_{\g_1}\tensor A_{\g_2}\]
called multiplication (fusion), respectively co-multiplication (fission).
\item Each point $x\in X$ determines a constant loop $\g_x$ at $x$. There are  unit and co-unit maps
\[ \eta:k \ra A_{\g_x}, \text{ and } \e:A_{\g_x}\ra k. \]
The counit maps give rise to the nondegeneracy condition: $\e\m$ is nondegenerate at each constant loop $\g_x$ in $X$. 
\item A {\it connection} on the bundle $A$ over $LX$, which assigns smoothly to any path $\G:I=[0,1]\ra LX$  a linear map $A_{\G_0}\ra A_{\G_1}$. This assignment maps a constant path to the identity, is compatible under decomposition of paths and satisfies the following strong invariance property: two paths in $LX$ that describe the same surface in $X$ give rise to the same parallel transport (two paths $\G,\G':I\ra LX$  {\it describe the same surface} if  their  adjoint maps $\check{\G},\check{\G}': I\times S^1\ra X$ are obtained one from the other by precomposition with a diffeomorphism of $I\times S^1$; compare with the definition of a superficial connection in \cite{Wal1}).

\end{itemize}

\noindent These data are subject to the following conditions:
\begin{enumerate}
\item (compatibility of fusion/fission with parallel transport) If $(1,2)$ and $(1',2')$ are two pairs of concatenated loops and $\G$ is a path from $1$ to $1'$ and $\G'$ is a path from $2$ to $2'$ so that for any  $t\in I$, $(\G(t), \G'(t))$  is a pair of concatenated loops, then the following diagrams commute
\[ \xymatrix  @C=3.2pc{ & 1\tensor 2 \ar[r]^-{\m_{12}} \ar[d]_{P(\G)\tensor P(\G')} & 12 \ar[d]^{P(\G\star\G')} &   & 12 \ar[r]^-{\n_{12}} \ar[d]_{P(\G\star\G')} & 1\tensor 2 \ar[d]^{P(\G)\tensor P(\G')} \\
& 1'\tensor 2' \ar[r]_-{\m_{1'2'}} & 1'2' &  &  1'2' \ar[r]_-{\n_{1'2'}} & 1'\tensor 2'  } \]
(We simplify notation and write $i$ instead of $A_i$.)

\item (associativity) Consider the following 3-petal loop (in $X$)
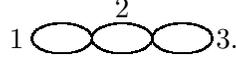
\[ \xy 
(0,0)*\ellipse(4,2){-}; 
(4,0)*\ellipse(4,2){-}; 
(-4,0)*\ellipse(4,2){-}; 
(-14,0)*+{1};
(14,0)*+{3.};
(0,4)*+{2};
\endxy \]
Then the diagram below commutes

\[ \xymatrix  @C=4pc{ 1\tensor 2\tensor 3 \ar[r]^{\m_{12}\tensor 3} \ar[d]_{1\tensor \m_{23}} & 12\tensor 3 \ar[d]^{\m_{(12)3}} \\
1\tensor 23 \ar[r]_{\m_{1(23)}} & 123.} \]

\item (co-associativity) Referring to the picture of the 3-petal loop above, the following diagram commutes
\[ \xymatrix  @C=4pc { 123 \ar[r]^{\n_{(12)3}} \ar[d]_{\n_{1(23)}} & 12\tensor 3 \ar[d]^{\n_{12}\tensor 3} \\
1\tensor 23 \ar[r]_{1\tensor \n_{23}} & 1\tensor 2\tensor 3.} \]

\item (compatibility of fusion and fission) The following diagram commutes
\[ \xymatrix @C=4pc{ 12\tensor 3 \ar[r]^{\m_{(12)3}} \ar[d]_{\n_{1(23)}} & 123 \ar[d]^{\n_{1(23)}} \\
1\tensor 2\tensor 3 \ar[r]_{\m_{1(23)}} & 1\tensor 23} \]
The diagram with the arrows reversed and $\m$'s and $\n$'s interchanged also commutes.

\item (compatibility of (co)units with parallel transport) For points $x, y$ in $X$, and $\g:I\ra X$ a path in $X$ connecting $x$ and $y$, the following diagrams commute
\[ \xymatrix{ & A_{\g_x}  \ar[dr]^{P(x;\g\bar{\g})} & & &  & A_{\g_x}  \ar[dr]^{\e_x} &\\
k \ar[ur]^{\eta_x} \ar[dr]_{\eta_y} &  &  A_{\g\bar{\g}} &  &  A_{\g\bar{\g}} \ar[ur]^{P(\g\bar{\g};x)} \ar[dr]_{P(\g\bar{\g};y)} &  &  k,\\
&  A_{\g_y} \ar[ur]_{P(y;\g\bar{\g})}   &  & &   &  A_{\g_y} \ar[ur]_{\e_y}   &    } \]
where $P(x;\g\bar{\g})$ and $P(y;\g\bar{\g})$ denote parallel transport from the loop at $x$, respectively at $y$, to  the loop $\g\bar{\g}$. 
Similarly, for $P(\g\bar{\g};x)$ and $P(\g\bar{\g};y)$.

\item (compatibility of units and fusion with parallel transport) Let $1$ denote a loop in $X$ based at $x\in X$, and let $\l_x$ denote the constant loop at $x$. The following diagram is commutative
\[ \xymatrix{ & 1\tensor \l_x \ar[dr]^{\m_{1x}} &\\
1 \ar[ur]^{\eta_x} \ar[rr]_{P(1;1\l_x)} & & 1\l_x} \]

\item (compatibility of counits and fission with parallel transport) With the notation of (6), the following diagram commutes
\[ \xymatrix{ & 1\tensor \l_x \ar[dr]^{\e_x} &\\
1\l_x \ar[ur]^{\n_{1x}} \ar[rr]_{P(1\l_x;1)} & & 1.} \]

\end{enumerate}

\end{defn}

% Homotopical FT
An easy consequence of the theorem above is the following

\begin{thm}  There is a 1-1 correspondence:

\[  \left\{
\begin{array}{l}
\text{2-dim homotopical field theories}\\
\text{\ \ \ \ \ \ \ \ \ \ \ \ \ over $X$}
\end{array} \right\}
\leftrightarrow \left\{
\begin{array}{l}
\text{Frobenius bundles with \textup{flat}}\\
\text{\ \  connections over $LX$}
\end{array} \right\} \]
\end{thm}

\begin{defn}A 2-dimensional field theory $E$ is {\it homotopical} if the definition \ref{2TFT} holds with the condition (3) of invariance under diffeomorphisms replaced by 
\begin{enumerate}
\item[(3')] (invariance under homotopies)
\[ E(\Si')= E(\Si), \]
whenever $\Si$ and $\Si'$ are smoothly homotopic.
\end{enumerate}
\end{defn}
A Frobenius bundle with a {\it flat} connection over $LX$ is a Frobenius bundle over $LX$ with connection so that the parallel transport is invariant under homotopies of paths in $LX$. Note that this is a stronger notion than the previous one, and it implies it. In particular, the invariance of parallel transport along paths in $LX$ that describe the same surface is automatically implied.\\

% Remarks

\noindent {\bf Remarks.} 1. The data of a Frobenius bundle with connection over $LX$ expresses the information contained in a 2-TFT over a space $X$ in a generators-and-relations type theorem.\\

\noindent 2. A Frobenius bundle over $LX$, when restricted to a constant loop at a point $x\in X$, encodes the information of a commutative Frobenius algebra. Thus, a Frobenius bundle over $LX$, when restricted to $X$, the space of constant loops in $LX$, gives rise to a bundle of commutative Frobenius algebras. We will see later (subsection \ref{FA}) that the fiber $A_\g$ over an arbitrary loop $\g$ in $X$ based at $x\in X$ admits an action of the Frobenius algebra $A_x$, the fiber over the constant loop at $x$. \\

\noindent 3. Property (5) in the definition above will allow us to define a field theory for surfaces in $X$ with no incoming or outgoing boundary, in particular the holonomy along closed surfaces in $X$. In fact, property (5) is equivalent to the following more general property 
\begin{enumerate}
\item[(5')] For points $x, y$ in $X$, and $\G:D^2\ra X$ a disk in $X$ containing $x$ and $y$, the following diagrams commute
\[ \xymatrix{ & A_{\g_x}  \ar[dr]^{P(x;\g)} & & &  & A_{\g_x}  \ar[dr]^{\e_x} &\\
k \ar[ur]^{\eta_x} \ar[dr]_{\eta_y} &  &  A_\g &  &  A_\g \ar[ur]^{P(x;\g)} \ar[dr]_{P(y;\g)} &  &  k,\\
&  A_{\g_y} \ar[ur]_{P(y;\g)}   &  & &   &  A_{\g_y} \ar[ur]_{\e_y}   &    } \]
where $\g$ is the restriction of $\G$ to the boundary $S^1=\partial D^2$, $P(x;\g)$ and $P(y;\g)$ denote parallel transport along $\G$ from the loop at $x$, respectively loop at $y$, to  the loop $\g$.
\end{enumerate}
This property immediately implies property (5) and it is obtained from (5) by applying further a parallel transport along paths in the loopspace.\\

\noindent 4. One can modify the definition of a 2-dimensional topological field theory over $X$ by requiring that the theory associates a vector space to any loop in $X$ up to reparametrization, i.e. to any {\it string} in $X$, and to any surface up to diffeomorphism (not necessarily identity on the boundary) in $X$ it associates a linear map between the fibers corresponding to the boundary. In this situation we obtain a similar description of 2-TFTs over $X$ as in theorem \ref{main}, where we replace the definition of a Frobenius bundle with connection over $LX$ by a simpler one in which we drop conditions (6) and (7) in the definition \ref{FB}, and the picture which gives rise to associativity and co-associativity of (2) and (3) is replaced by an ``honest" 3-petal loop, i.e. the three loops share a common point. \\

\noindent\textbf{Higher (co)associativity.} \label{higher} The compatibility of parallel transport with fusion/fission and the (co)associativity in the definition above implies the higher (co)associativity. Consider, for example, the following like-loop (in $X$):
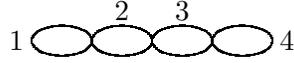
\[ \xy 
(2,0)*\ellipse(4,2){-}; 
(-2,0)*\ellipse(4,2){-}; 
(-6,0)*\ellipse(4,2){-}; 
(6,0)*\ellipse(4,2){-}; 
(-18,0)*+{1};
(4,4)*+{3};
(18,0)*+{4};
(-4,4)*+{2};
\endxy \]

\vspace{.1in}

\noindent There are three ways to break the big loop going around the loops 1, 2, 3 and 4 into the four little loops using the fission maps and the coassociativity of the fission maps. Let us show that this is independent of the possible choices.
For that, let us look at the following diagrams:

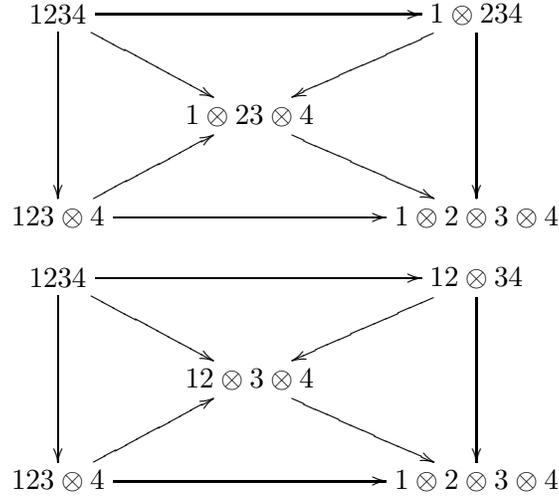
\[ \xymatrix{ 1234 \ar[rr] \ar[dd] \ar[dr] & & 1\tensor 234 \ar[dd] \ar[dl]  \\
& 1\tensor 23 \tensor 4 \ar[dr] & \\
123\tensor 4 \ar[rr] \ar[ur] & & 1\tensor 2\tensor 3 \tensor 4} \]

\[ \xymatrix{ 1234 \ar[rr] \ar[dd] \ar[dr] & & 12\tensor 34 \ar[dd] \ar[dl] \\
& 12\tensor 3 \tensor 4 \ar[dr] & \\
123\tensor 4 \ar[rr] \ar[ur] & & 1\tensor 2\tensor 3 \tensor 4} \]

\noindent where the maps are the obvious fission maps combined possibly with canonical reparametrizations of the loops. The inside-the-square diagrams are commutative making the outer squares commutative. This shows independence on the various choices to ``arrive'' from 1234 to the little loops labelled 1 through 4. One can proceed now by induction to show the higher co-associativity. In a similar manner, one deals with the higher associativity of the fusion product. Commutativity of fusion and fission implies the commutativity of the above higher fusion and fission.

% The proof
\section{Proof of the theorem}
\noindent ``$\rightarrow$" Let us  first show how a 2-TFT $E$ over $X$ gives rise to a Frobenius bundle with connection over $LX$. For a loop $\g$ in $X$, define $A_\g=E(\g)$. If we consider the smooth family of loops in $X$ parametrized by the ``universal loop space" $LX$, the field theory provides a family of vector spaces $\{A_\g\}$ parametrized by $LX$, i.e. a vector bundle $A$ over $LX$. We should endow the bundle $A$ with a multiplication and a comultiplication map and a connection.\\

\noindent {\it (Co)multiplication.} Let  $\g=\g_1\star\g_2$ be an 8-like loop in $X$. Define the comultiplication map
\[ \n: A_\g\ra A_{\g_1}\tensor A_{\g_2}, \ \ \ \n:= E(\Si), \]
where $\Si$ is the map from the pair of pants into $X$ such that at the top end restricts to $\g$ and at the bottom two ends restricts to $\g_1\amalg \g_2$ (see picture).

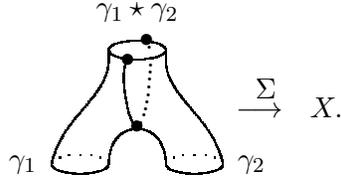
\[ \xy 
0;/r.30pc/:
(0,4)*{\g_1\star\g_2};
(-12,-12)*{\g_1};
(12,-12)*{\g_2};
(0,0)*\ellipse(3,1){-}; 
(-3,-6)*\ellipse(3,1){.}; 
(3,-6)*\ellipse(3,1){.}; 
(-3,-6)*\ellipse(3,1)__,=:a(-180){-}; 
(3,-6)*\ellipse(3,1)__,=:a(-180){-}; 
(-3,-12)*{}="1"; 
(3,-12)*{}="2"; 
(-9,-12)*{}="A2"; 
(9,-12)*{}="B2"; 
"1";"2" **\crv{(-3,-7) & (3,-7)}; 
(-3,0)*{}="A"; 
(3,0)*{}="B"; 
(-3,-1)*{}="A1"; 
(3,-1)*{}="B1"; 
(0,-8)*{\bullet};
(-1,-1)*{\bullet};
(1,1)*{\bullet};
(-1,-1);(0,-8) **\crv{(-2,-2) & (-1,-8)};
(1,1);(0,-8) **\crv{~*=<3pt>{.}(2,0) & (1,-8)};
"A";"A1" **\dir{-}; 
"B";"B1" **\dir{-}; 
"B2";"B1" **\crv{(8,-7) & (3,-5)}; 
"A2";"A1" **\crv{(-8,-7) & (-3,-5)}; 
(16,-6)*{ \longrightarrow \ \ X.};
(13.5,-4)*{\Si};
\endxy \]

\vspace{.2in}

\noindent Similarly, define the multiplication map
\[ \m: A_{\g_1}\tensor A_{\g_2} \ra A_\g, \ \ \ \m:=E(\bar{\Si}) \]
where $\bar{\Si}$ is the map $\Si$  from the pair of pants with reversed orientation (read down-up) into $X$. \\

\noindent {\it (Co)units.} Let $x\in X$ be an arbitrary point and $\Si_x: D^2\ra X$ be the constant map with value $x$, viewed as a cobordism in $X$ from the empty set to the constant loop at $x$. Define the unit
\[ \eta_x:k\ra A_x, \ \ \ \eta_x:= E(\Si_x). \]
Similarly, define the co-unit 
\[ \e_x:A_x \ra k, \ \ \ \e_x:= E(\bar{\Si}_x). \]
where $\bar{\Si}_x:\bar{D}^2\ra X$ is the map $\Si_x$  viewed as a cobordism in $X$ from the constant loop at $x$  to the empty set. \\

\noindent {\it Connection.} Let now $\G:I\ra LX$ be a path in the loop space between $\g_1$ and $\g_2$. We can interpret $\G$ as a map $\Si:I\times S^1\ra X$, i.e. as a bordism between $\g_1$ and $\g_2$. Define 
\[P(\G): A_{\g_1}\ra A_{\g_2}, \ \ \ P(\G):=E(\Si). \]
The map $P$ defined on the path space of $LX$ satisfies the following properties:

\begin{enumerate}
\item $P(\G_\g)=1_{A_\g}$, where $\G_\g$ is a constant path at $\g\in LX$.

\item (Invariance under diffeomorphisms) $P(\G')=P(\G)$, where $\G':I\ra LX$ is the path in $LX$, whose adjoint map $\check{\G}':I\times S^1\ra X$ is given by $\check{\G}'= \check{\G}\comp \Phi$, where $\Phi:I\times S^1 \ra I\times S^1$ is an  arbitrary diffeomorphism which is the identity on the boundary.

\item (Compatibility under decomposition) $P(\G)= P(\G_2)\comp P(\G_1)$,
if $\G$ decomposes as $\G=\G_2\star\G_1$.
\end{enumerate}
These are exactly the properties that define a connection (see the definition above) on the bundle $A$ over $LX$, i.e. parallel transport along paths in $LX$, or along cylinders in $X$.

The properties (1)-(3) are easy to see. Property (4) expressing the compatibility of fusion and fission is
 a consequence of the following diffeomorphism of surfaces (read up-down and down-up):
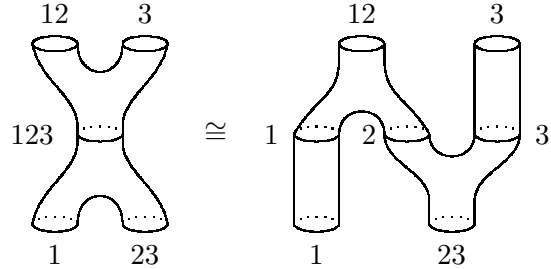
\[
\xy 
(-6,16)*{12};
(6,16)*{3};
(-6,-16)*{1};
(6,-16)*{23};
(-9,0)*{123};
(0,0)*\ellipse(3,1){.}; 
(0,0)*\ellipse(3,1)__,=:a(-180){-}; 
(-3,-6)*\ellipse(3,1){.}; 
(3,-6)*\ellipse(3,1){.}; 
(-3,-6)*\ellipse(3,1)__,=:a(-180){-}; 
(3,-6)*\ellipse(3,1)__,=:a(-180){-}; 
(-3,-12)*{}="1"; 
(3,-12)*{}="2"; 
(-9,-12)*{}="A2"; 
(9,-12)*{}="B2"; 
"1";"2" **\crv{(-3,-7) & (3,-7)}; 
(-3,0)*{}="A"; 
(3,0)*{}="B"; 
(-3,-1)*{}="A1"; 
(3,-1)*{}="B1"; 
"A";"A1" **\dir{-}; 
"B";"B1" **\dir{-}; 
"B2";"B1" **\crv{(8,-7) & (3,-5)}; 
"A2";"A1" **\crv{(-8,-7) & (-3,-5)}; 
%REFLECT 
(-3,6)*\ellipse(3,1){-}; 
(3,6)*\ellipse(3,1){-}; 
(-3,12)*{}="1"; 
(3,12)*{}="2"; 
(-9,12)*{}="A2"; 
(9,12)*{}="B2"; 
"1";"2" **\crv{(-3,7) & (3,7)}; 
(-3,0)*{}="A"; 
(3,0)*{}="B"; 
(-3,1)*{}="A1"; 
(3,1)*{}="B1"; 
"A";"A1" **\dir{-}; 
"B";"B1" **\dir{-}; 
"B2";"B1" **\crv{(8,7) & (3,5)}; 
"A2";"A1" **\crv{(-8,7) & (-3,5)}; 
\endxy
\quad  \cong \quad
 \xy 
 (0,16)*{12};
(18,16)*{3};
(-12,0)*{1};
(1,0)*{2};
(24,0)*{3};
 (-6,-16)*{1};
(12,-16)*{23};
(0,6)*\ellipse(3,1){-}; 
(-3,0)*\ellipse(3,1){.}; 
(-3,0)*\ellipse(3,1)__,=:a(-180){-}; 
(3,0)*\ellipse(3,1){.}; 
(3,0)*\ellipse(3,1)__,=:a(-180){-}; 
(9,0)*\ellipse(3,1){.}; 
(9,0)*\ellipse(3,1)__,=:a(-180){-}; 
(6,-6)*\ellipse(3,1){.}; 
(6,-6)*\ellipse(3,1)__,=:a(-180){-}; (-3,0)*{}="1"; 
(3,0)*{}="2"; 
(-9,0)*{}="A2"; 
(9,0)*{}="B2"; 
"1";"2" **\crv{(-3,4) & (3,4)}; 
(-3,12)*{}="A"; 
(3,12)*{}="B"; 
(-3,10)*{}="A1"; 
(3,10)*{}="B1"; 
"A";"A1" **\dir{-}; 
"B";"B1" **\dir{-}; 
"B2";"B1" **\crv{(6,5) & (3,5)}; 
"A2";"A1" **\crv{(-6,5) & (-3,5)}; 
(-3,-6)*\ellipse(3,1){.}; 
(-3,-6)*\ellipse(3,1)__,=:a(-180){-}; 
(-9,-12)*{}="A3"; 
(-3,-12)*{}="3";
"A2";"A3" **\dir{-}; 
"1";"3" **\dir{-}; 
(9,-9)*{}="C1"; 
(15,-9)*{}="D1"; 
(9,-12)*{}="A4"; 
(15,-12)*{}="4"; 
"A4";"C1" **\dir{-}; 
"4";"D1" **\dir{-}; 
(15,0)*{}="5"; 
(21,0)*{}="6"; 
(15,12)*{}="7"; 
(21,12)*{}="8"; 
(9,6)*\ellipse(3,1){-};
"5";"7" **\dir{-}; 
"6";"8" **\dir{-}; 
"C1";"2" **\crv{(9,-5) & (3,-4)}; 
"D1";"6" **\crv{(15,-5) & (21,-4)}; 
"B2";"5" **\crv{(9,-4) & (15,-4)}; 
\endxy \]

\vspace{.1in}

\noindent These surfaces map in the indicated way to the 3-petal loops. 

Property (5) is a consequence of the following two observations. First, if $\Si:Z\ra X$ is a surface in $X$ that is constant on a neighborhood about a point $z\in Z$, and $Z'$ is obtained from $Z$ by collapsing the neighborhood to a point and $\Si'$ is the induced map, then $E(\Si)=E(\Si')$. This happens because the field theory is smooth, and $\Si'$ is a limit (as $t\ra 0$) of maps $\Si_t$ related by difeomorphisms to $\Si$, so
\[ E(\Si_t)= E(\Si) \lra E(\Si'), \ \ {\text as }\  t\lra 0. \] 
Second, we notice that the bordisms in $X$ in the picture below are diffeomorphic

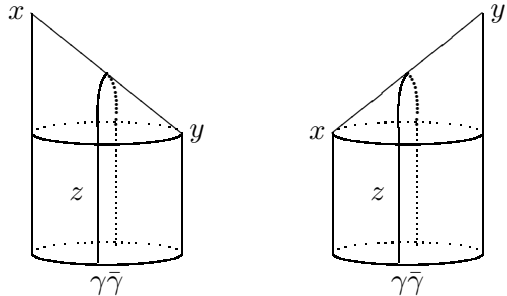
\[ \xy 
 (-52,24)*{x};
(-40,-12)*{\g\bar{\g}};
(-28,8)*{y};
(-44,0)*{z};
(-20,4)*\ellipse(10,1.5){.}; 
(-20,-4)*\ellipse(10,1.5){.}; 
(-20,4)*\ellipse(10,1.5)__,=:a(-180){-}; 
(-20,-4)*\ellipse(10,1.5)__,=:a(-180){-}; 
(-50,8)*{}="TL"; 
(-30,8)*{}="TR"; 
(-50,-8)*{}="BL"; 
(-30,-8)*{}="BR"; 
(-50,24)*{}="x";
(-40,8)*{}="z";
(-40,16)*{}="c";
"z"+(-1.2,-1.2)="a";
"z"+(1.2,1.2)="b";
%"a"*+{\bullet};
"a";"c" **\crv{(-41,7) & (-42,14)}; 
"b";"c" **\crv{~*=<2pt>{.}(-39,9) & (-38,14)}; 
"x"; "BL" **\dir{-}; 
"x"; "TR" **\dir{-}; 
"TR"; "BR" **\dir{-}; 
"a"; (-41.2,-9.2) **\dir{-}; 
"b"; (-38.8,-6.8) **\dir{.};
\endxy
\xy
 (52,24)*{y};
(40,-12)*{\g\bar{\g}};
(28,8)*{x};
(36,0)*{z};
(20,4)*\ellipse(10,1.5){.}; 
(20,-4)*\ellipse(10,1.5){.}; 
(20,4)*\ellipse(10,1.5)__,=:a(-180){-}; 
(20,-4)*\ellipse(10,1.5)__,=:a(-180){-}; 
(50,8)*{}="TL"; 
(30,8)*{}="TR"; 
(50,-8)*{}="BL"; 
(30,-8)*{}="BR"; 
(50,24)*{}="x";
(40,8)*{}="z";
(40,16)*{}="c";
"z"+(-1.2,-1.2)="a";
"z"+(1.2,1.2)="b";
%"a"*+{\bullet};
"a";"c" **\crv{(39,7) & (38,14)}; 
"b";"c" **\crv{~*=<2pt>{.}(41,9) & (42,14)}; 
"x"; "BL" **\dir{-}; 
"x"; "TR" **\dir{-}; 
"TR"; "BR" **\dir{-}; 
"a"; (38.8,-9.2) **\dir{-}; 
"b"; (41.2,-6.8) **\dir{.};
\endxy \]
The two bordisms in $X$ drawn above are constant along vertical planes perpendicular to the plane of the paper; for example, along the curves labeled by $z$, the maps are constant equal to $z$, where $z$ is a point on the path $\g$ joining $x$ and $y$. The two bordisms (read up-down) give rise to the two compositions $k\ra  A_{\g\bar{\g}}$ appearing in (5). Read down-up, they give rise to the two compositions $A_{\g\bar{\g}}\ra k$.

Properties (6) and (7) are easy to see, reflecting the diffeomorphism in the  picture below (read up-down and down-up) in the case of a Frobenius algebra (or a 2-TFT over a point)

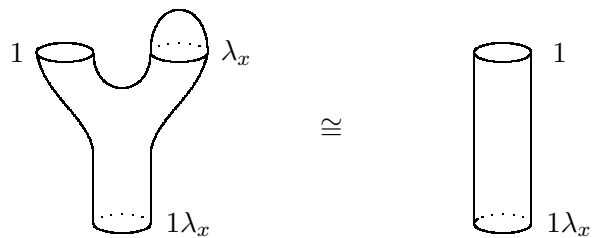
\[ \xy 
0;/r.30pc/:
(-11,12)*{1};
(46,12)*{1};
(22,4)*{\cong};
(12,12)*{\l_x};
(7,-6)*{1\l_x};
(47,-6)*{1\l_x};
(0,-3)*\ellipse(3,1){.}; 
(0,-3)*\ellipse(3,1)__,=:a(-180){-}; 
(-3,6)*\ellipse(3,1){-}; 
(3,6)*\ellipse(3,1){.}; 
(3,6)*\ellipse(3,1)__,=:a(-180){-}; 
(-3,12)*{}="1"; %(inside 
(3,12)*{}="2"; %top ellipses) 
(-9,12)*{}="A2";%(Outside 
(9,12)*{}="B2"; %top ellipses) 
"1";"2" **\crv{(-3,7) & (3,7)};%(Top curve) 
"2";"B2" **\crv{(3,18) & (9,18)};
(-3,-6)*{}="A";%(sides of 
(3,-6)*{}="B"; % bottom ellipse) 
(-3,1)*{}="A1";%(end points 
(3,1)*{}="B1"; % of straight lines) 
"A";"A1" **\dir{-}; 
"B";"B1" **\dir{-}; %straight lines 
"B2";"B1" **\crv{(8,7) & (3,5)}; 
"A2";"A1" **\crv{(-8,7) & (-3,5)}; 
(20,6)*\ellipse(3,1){-}; 
(20,-3)*\ellipse(3,1){.}; 
(20,-3)*\ellipse(3,1)_,=:a(220){-}; 
(37,12)*{}="TL"; 
(43,12)*{}="TR"; 
(37,-6)*{}="BL"; 
(43,-6)*{}="BR"; 
"TL"; "BL" **\dir{-}; 
"TR"; "BR" **\dir{-}; 
\endxy \]

\vspace{.1in}

\noindent This ends one direction in the proof of the theorem.\\ 

\noindent ``$\leftarrow$" Conversely, start with a Frobenius bundle $A$ with connection over $LX$. We would like to produce a 2-TFT over $X$. To each loop $\g$ in $X$ we associate the vector space $E(\g):= A_\g$. The field theory is determined by specifying the linear maps corresponding to bordisms between loops in $X$. Each such bordism in $X$ can be recovered (via gluing) from the following ``basic" bordisms in $X$ (read up-down):

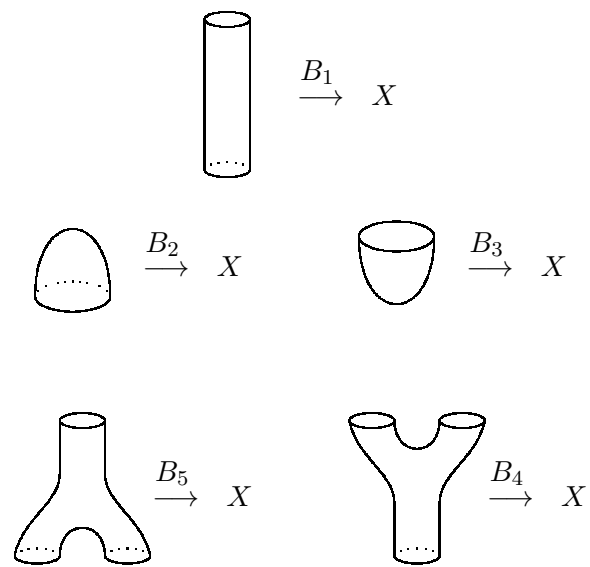
\[ \xy 
(0,5)*\ellipse(3,1){-}; 
(0,-5)*\ellipse(3,1){.}; 
(0,-5)*\ellipse(3,1)__,=:a(-180){-}; 
(-3,10)*{}="TL"; 
(3,10)*{}="TR"; 
(-3,-10)*{}="BL"; 
(3,-10)*{}="BR"; 
"TL"; "BL" **\dir{-}; 
"TR"; "BR" **\dir{-}; 
(16,0)*{ \longrightarrow \ \ X};
(12,3)*{B_1}
\endxy \]

\[ \xy 
(0,-3)*\ellipse(5,2){.}; 
(0,-3)*\ellipse(5,2)__,=:a(-180){-}; 
(-5,-6)*{}="TL"; 
(5,-6)*{}="TR"; 
"TL";"TR" **\crv{(-5,6) & (5,6) }; 
(16,-2)*{ \longrightarrow \ \ X};
(12,1)*{B_2}
\endxy 
\quad \quad \quad \quad
 \xy 
 (0,1)*\ellipse(5,2){-}; 
(-5,2)*{}="TL"; 
(5,2)*{}="TR"; 
"TL";"TR" **\crv{(-5,-10) & (5,-10) }; 
(16,-2)*{ \longrightarrow \ \ X};
(12,1)*{B_3}
\endxy \]

\vspace{.3in}

\[ \xy 
(0,3)*\ellipse(3,1){-}; 
(-3,-6)*\ellipse(3,1){.}; 
(-3,-6)*\ellipse(3,1)__,=:a(-180){-}; 
(3,-6)*\ellipse(3,1){.}; 
(3,-6)*\ellipse(3,1)__,=:a(-180){-}; 
%(-1.75,-4)*\ellipse(3.5,.8){.}; 
%(-1.75,-4)*\ellipse(3.5,.8)__,=:a(-180){-}; 
%(1.75,-4)*\ellipse(3.5,.8){.}; 
%(1.75,-4)*\ellipse(3.5,.8)__,=:a(-180){-}; 
(-3,-12)*{}="1"; 
(3,-12)*{}="2"; 
(-9,-12)*{}="A2"; 
(9,-12)*{}="B2"; 
"1";"2" **\crv{(-3,-7) & (3,-7)}; 
(-3,6)*{}="A"; 
(3,6)*{}="B"; 
(-3,-1)*{}="A1"; 
(3,-1)*{}="B1"; 
"A";"A1" **\dir{-}; 
"B";"B1" **\dir{-}; 
"B2";"B1" **\crv{(8,-7) & (3,-5)}; 
"A2";"A1" **\crv{(-8,-7) & (-3,-5)}; 
(16,-4)*{ \longrightarrow \ \ X};
(12,-1)*{B_5};
(0,-8)*{}="C";
(-7,-8)*{}="D";
%"D";"C"**\dir{-}
\endxy 
\quad \quad \quad \ 
\xy 
(0,-6)*\ellipse(3,1){.}; 
(0,-6)*\ellipse(3,1)__,=:a(-180){-}; 
(-3,3)*\ellipse(3,1){-}; 
(3,3)*\ellipse(3,1){-}; 
(-3,6)*{}="1"; %(inside 
(3,6)*{}="2"; %top ellipses) 
(-9,6)*{}="A2";%(Outside 
(9,6)*{}="B2"; %top ellipses) 
"1";"2" **\crv{(-3,1) & (3,1)};%(Top curve) 
(-3,-12)*{}="A";%(sides of 
(3,-12)*{}="B"; % bottom ellipse) 
(-3,-5)*{}="A1";%(end points 
(3,-5)*{}="B1"; % of straight lines) 
"A";"A1" **\dir{-}; 
"B";"B1" **\dir{-}; %straight lines 
"B2";"B1" **\crv{(8,1) & (3,-1)}; 
"A2";"A1" **\crv{(-8,1) & (-3,-1)}; 
(16,-4)*{ \longrightarrow \ \ X};
(12,-1)*{B_4}
\endxy \]

\vspace{.3in}

We shall specify the functor $E$ on such bordisms, and then show that, for an arbitrary bordism, $E$ is independent of the various decompositions of the bordism into basic bordisms. Let $P$ denote the parallel transport map along paths in the loopspace $LX$ determined by the connection on the bundle $A$ over $LX$.\\

%\begin{itemize}
\noindent \textbf{$B_1$-type} bordism: let $\g_1$, $\g_2$ denote the top loop, respectively the bottom loop, of the bordism $B_1$. Define 
\[ E(B_1): A_{\g_1}\ra A_{\g_2},\ \ \ E(B_1):= P(\G), \]
where $\G$ is the path in the loop space determined by the cylinder, connecting $\g_1$ and$\g_2$. \\

\noindent \textbf{$B_2$-type} bordism: let $\g$ denote the bottom loop of $B_2$, and let a point on the surface (that does not lie on the boundary)  mapping to $x\in X$. Define
\[ E(B_2): k \ra A_\g, \ \ \ E(B_2):=P(\G_x)\comp \eta_x, \]
where $\eta_x: k\ra A_x$ is the unit of the Frobenius algebra $A_x$ and $\G_x$ is the map from the surface into $X$, viewed as a path in the loopspace $LX$ from the constant path at $x$ to the loop $\g$. This is independent of the various choices since the unit structure maps are compatible with the parallel transport, by the property (5) in the definition of a Frobenius bundle with connection, or its equivalent (5'). \\

\noindent \textbf{$B_3$-type} bordism: let $\g$ denote the top loop of $B_3$, and let a point on the surface as before that maps to $x\in X$. Define
\[ E(B_3): A_\g \ra k, \ \ \ E(B_3):=\e_x\comp P(\G_x), \]
where $\e_x: A_x\ra k$ is the counit of the Frobenius algebra $A_x$ and $\G_x$ is the map from the surface into $X$, viewed as a path in $LX$ from the loop $\g$ to the constant loop at $x$. This is independent of the various choices since the counit structure maps are compatible with the parallel transport by property (5').\\ 

\noindent Before we proceed to define the field theory for the pairs of pants of the type $B_4$ and $B_5$ we describe a special type of interaction between two loops in $X$. Specifically, consider the following picture
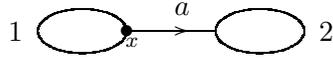
\[ \xy 
0;/r.35pc/:
(-4,0);(4,0) **\dir{-}?(.7)*\dir{>}; 
(4,0)*\ellipse(4,2){-};
(-4,0)*\ellipse(4,2){-}; 
(-14,0)*+{1};
(14,0)*+{2};
(1,2)*+{a};
(-4,0)*+{\bullet};
(-3.5,-1)*+{\scriptstyle x}
\endxy \]

\noindent consisting of two loops in $X$ and a path between the loops, labeled respectively $1$, $2$ and $a$. The basepoint of the loop $1$ maps to the point $x\in X$. Denote by $\l_x$ the constant loop at $x$. 
In the diagrams below

\[ \xymatrix{ 1\tensor aa^{-1} \ar[rr]^\m  & & 1aa^{-1}  & &  1\tensor aa^{-1} \ar[dr]^\e \ar[dd]_P & & 1aa^{-1}  \ar[ll] _\n\ar[dd]^P \ar[dl]_P\\
& 1 \ar[ul]_\eta \ar[ur]^P \ar[dl]_\eta \ar[dr]^P & & &  & 1  &\\
1\tensor \l_x \ar[rr]_\m \ar[uu]^P& & 1\l_x \ar[uu]_P  & &   1\tensor \l_x \ar[ur]^\e & & 1\l_x \ar[ll]^\n \ar[ul]_P } \]
we would like to say that the upper triangles commute, i.e. $\m\eta=P$, respectively $\e\n=P$. This is true for the lower triangles, by the compatibility of units and fusion, respectively counits and fission with parallel transport. The left-hand side triangles commute since (co)units are compatible with parallel transport. The right-hand  side triangles also commute since parallel transport is compatible with gluing of paths. The outside squares commute since (co)multiplication is compatible with parallel transport. This gives the required commutativity.

From this we obtain the following diagrams

\[ \xymatrix{ 1\tensor 2 \ar[rr]^{P\tensor 2} \ar[dd]_{1\tensor P} \ar[dr]_\eta & & 1aa^{-1}\tensor 2 \ar[dd]_\m  \\
& 1\tensor aa^{-1} \tensor 2 \ar[ur]_{\m\tensor 2} \ar[dl]^{1\tensor \m} & \\
1\tensor aa^{-1}2 \ar[rr]_\m & & 1aa^{-1}2} \]

\[ \xymatrix{ 1\tensor 2  & & 1aa^{-1}\tensor 2 \ar[ll]_{P\tensor 2} \ar[dl]^{\n\tensor 2} \\
& 1\tensor aa^{-1} \tensor 2 \ar[ul]^\e  & \\
1\tensor aa^{-1}2 \ar[ur]^{1\tensor \n} \ar[uu]^{1\tensor P} & & 1aa^{-1}2 \ar[ll]_\n \ar[uu]_\n }\]
with all the inside n-gons commutative, making the outer square commutative. The commutativity of these diagrams will allow us to say that the field theory is well defined for pairs of pants mapping into the space $X$. \\

\noindent \textbf{$B_4$-type} bordism:
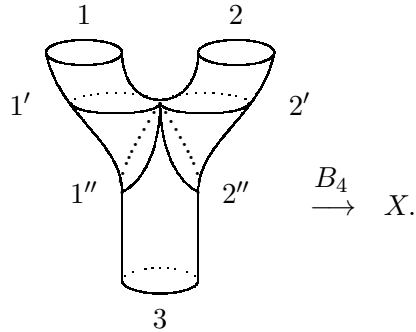
\[ \xy 
0;/r.40pc/:
(-6,9)*{1};
(6,9)*{2};
(0,-15)*{3};
(-11,2)*{1'};
(11,2)*{2'};
(-6,-5)*{1''};
(6,-5)*{2''};
(0,-6)*\ellipse(3,1){.}; 
(0,-6)*\ellipse(3,1)__,=:a(-180){-}; 
(-3,3)*\ellipse(3,1){-}; 
(3,3)*\ellipse(3,1){-}; 
(-1.75,1)*\ellipse(3.5,.8){.}; 
(-1.75,1)*\ellipse(3.5,.8)__,=:a(-180){-}; 
(1.75,1)*\ellipse(3.5,.8){.}; 
(1.75,1)*\ellipse(3.5,.8)__,=:a(-180){-}; 
(-3,6)*{}="1"; %(inside 
(3,6)*{}="2"; %top ellipses) 
(-9,6)*{}="A2";%(Outside 
(9,6)*{}="B2"; %top ellipses) 
"1";"2" **\crv{(-3,1) & (3,1)};%(Top curve) 
(-3,-12)*{}="A";%(sides of 
(3,-12)*{}="B"; % bottom ellipse) 
(-3,-5)*{}="A1";%(end points 
(3,-5)*{}="B1"; % of straight lines) 
"A";"A1" **\dir{-}; 
"B";"B1" **\dir{-}; %straight lines 
"B2";"B1" **\crv{(8,1) & (3,-1)}; 
"A2";"A1" **\crv{(-8,1) & (-3,-1)}; 
(16,-6)*{ \longrightarrow \ \ X.};
(13.5,-4)*{B_4};
(0,2)*{}="C";
%"B1";"C" **\dir{-};
"C";"B1" **\crv{(0,-1) & (1,-4)}; 
"C";"A1" **\crv{(0,-1) & (-1,-4)}; 
"C";"B1" **\crv{~*=<3pt>{.}(3,2),(3.8,-4)};
"C";"A1" **\crv{~*=<3pt>{.}(-3,2),(-3.8,-4)};
\endxy \]

%\vspace{.1in}

\noindent Let $\g_1$ and $\g_2$ denote the top two loops and  $\g_3$ the bottom loop. Let $\g_{1'}$ and $\g_{2'}$ be the loops corresponding to the circles labelled $1'$ and $2'$ in the picture. Define
$E(B_4): A_1\tensor A_2\ra A_3$ by
\[ E(B_4):= P(\Si_{(1'2')3})\comp\m_{1'2'} \comp P(\Si_{11'})\tensor P(\Si_{22'}), \]
where $\Si_{(1'2')3}$ is the bordism in $X$ between $\g_{1'}\star\g_{2'}$ and $\g_3$, $\m_{1'2'}$ is the multiplication map determined by $\g_{1'}\star\g_{2'}$, $\Si_{11'}$ is the bordism in $X$ between $\g_{1}$ and $\g_{1'}$ and $\Si_{22'}$ is the bordism in $X$ between $\g_{2}$ and $\g_{2'}$. 

The definition is independent on the choice of the intermediate 8-like loop. Indeed, let us consider first the case when the intermediate 8-like loops share a common point as the loops $1'2'$ and $1''2''$ in the picture above (we shall simplify the notation by writing for example simply 1 for the fiber $A_{1}$ over the loop $\g_1$ and $P(11')$ for the parallel transport along the bordism $\Si_{11'}$ etc., when no possibility of confusion arises). In the diagram below

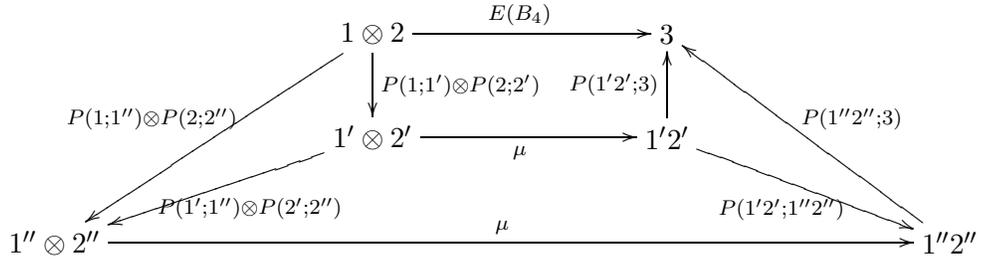
\[ \xymatrix @C=6.8pc { & 1\tensor 2 \ar[r]^{E(B_4)} \ar[d]^{P(1;1')\tensor P(2;2')}  \ar[ddl]_{P(1;1'')\tensor P(2;2'')} & 3 & \\
& 1'\tensor 2' \ar[r]_\m \ar[dl]^-{P(1';1'')\tensor P(2';2'')} & 1'2' \ar[dr]_{P(1'2';1''2'')} \ar[u]^-{ P(1'2';3)} & \\
1''\tensor 2'' \ar[rrr]^\m & & & 1''2'' \ar[uul]_{P(1''2'';3)} }\]
all the inside diagrams are commuting  expressing either the compatibility of the parallel transport under gluing of paths (of loops) or the compatibility of fusion and parallel transport. This shows that $E(B_4)$ is well defined in this case.

In the situation when the two 8-like loops do not share their junction point (see the picture below), we proceed as follows.

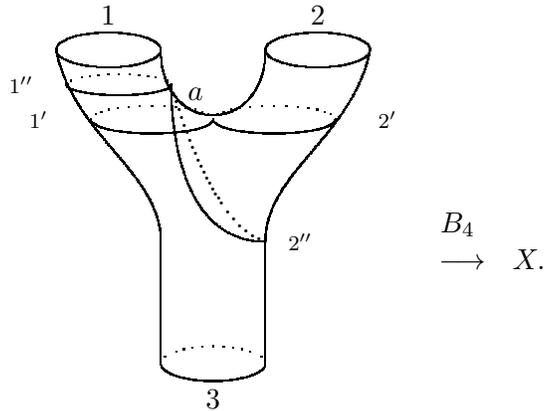
\[ \xy 
0;/r.55pc/:
(-6,8)*{1};
(6,8)*{2};
(0,-14)*{3};
(-10,2)*+{\scriptstyle 1'};
(-1,3.5)*+{a};
(10,2)*+{\scriptstyle 2'};
(-11,4)*+{\scriptstyle 1''};
(5,-5)*+{\scriptstyle 2''};
(0,-6)*\ellipse(3,1){.}; 
(0,-6)*\ellipse(3,1)__,=:a(-180){-}; 
(-3,3)*\ellipse(3,1){-}; 
(3,3)*\ellipse(3,1){-}; 
(-1.75,1)*\ellipse(3.5,.8){.}; 
(-1.75,1)*\ellipse(3.5,.8)__,=:a(-180){-}; 
(1.75,1)*\ellipse(3.5,.8){.}; 
(1.75,1)*\ellipse(3.5,.8)__,=:a(-180){-}; 
(-2.7,2)*\ellipse(3,.6){.}; 
(-2.7,2)*\ellipse(3,.6)__,=:a(-180){-}; 
(-3,6)*{}="1"; %(inside 
(3,6)*{}="2"; %top ellipses) 
(-9,6)*{}="A2";%(Outside 
(9,6)*{}="B2"; %top ellipses) 
"1";"2" **\crv{(-3,1) & (3,1)};%(Top curve) 
(-3,-12)*{}="A";%(sides of 
(3,-12)*{}="B"; % bottom ellipse) 
(-3,-5)*{}="A1";%(end points 
(3,-5)*{}="B1"; % of straight lines) 
"A";"A1" **\dir{-}; 
"B";"B1" **\dir{-}; %straight lines 
"B2";"B1" **\crv{(8,1) & (3,-1)}; 
"A2";"A1" **\crv{(-8,1) & (-3,-1)}; 
(16,-6)*{ \longrightarrow \ \ X.};
(14,-4)*{B_4};
(0,2)*{}="C";
(-2.4,4)*{}="D";
"D";"B1" **\crv{(-2.4,3) & (-2.4,-5)}; 
"D";"B1" **\crv{~*=<3pt>{.}(-1,2),(0,-4)};
\endxy \]

\vspace{.1in}

\noindent Let $a$ denote the path (in $X$) between the junction points of the two 8-like loops. Again, in the diagram below
\[ \xymatrix{1\tensor 2 \ar[r]^-{P\tensor P} \ar[d]_{P\tensor P} & 1''\tensor a^{-1}a2' \ar[r]^-P \ar[d]^\m & 1''\tensor 2'' \ar[d]^\m \\
1''aa^{-1}\tensor 2' \ar[r]^\m \ar[d]_P & 1''aa^{-1}2' \ar[r]^-P \ar[d]^P & 1''2'' \ar[d]^P \\
1'\tensor 2' \ar[r]_\m & 1'2' \ar[r]_P & 3} \]
all the inside squares commute. Indeed, the upper left square commutes by the considerations made before defining the field theory for bordisms of type $B_4$. The upper right and the lower left square commute by the compatibility of fusion and parallel transport. Finally, the lower right square commutes since parallel transport is compatible under glueing of paths. Reading off the commutativity of the outer square, we obtain the desired independence.\\

%\[ \xy 
%(-6,10)*{1};
%(6,10)*{2};
%(0,-16)*{3};
%(-11,2)*{1'};
%(11,2)*{2'};
%(0,-6)*\ellipse(3,1){.}; 
%(0,-6)*\ellipse(3,1)__,=:a(-180){-}; 
%(-3,3)*\ellipse(3,1){-}; 
%(3,3)*\ellipse(3,1){-}; 
%(-1.75,1)*\ellipse(3.5,.8){.}; 
%(-1.75,1)*\ellipse(3.5,.8)__,=:a(-180){-}; 
%(1.75,1)*\ellipse(3.5,.8){.}; 
%(1.75,1)*\ellipse(3.5,.8)__,=:a(-180){-}; 
%(-3,6)*{}="1"; %(inside 
%(3,6)*{}="2"; %top ellipses) 
%(-9,6)*{}="A2";%(Outside 
%(9,6)*{}="B2"; %top ellipses) 
%"1";"2" **\crv{(-3,1) & (3,1)};%(Top curve) 
%(-3,-12)*{}="A";%(sides of 
%(3,-12)*{}="B"; % bottom ellipse) 
%(-3,-5)*{}="A1";%(end points 
%(3,-5)*{}="B1"; % of straight lines) 
%"A";"A1" **\dir{-}; 
%"B";"B1" **\dir{-}; %straight lines 
%"B2";"B1" **\crv{(8,1) & (3,-1)}; 
%"A2";"A1" **\crv{(-8,1) & (-3,-1)}; 
%(16,-6)*{ \longrightarrow \ \ X};
%(12,-4)*{B_4};
%(0,2)*{}="C";
%%"B";"C" **\dir{-};
%\endxy \]

%\vspace{.1in}

%\noindent Let $\g_1$ and $\g_2$ be the top two loops and  $\g_3$ be the bottom loop. Let $\g_{1'}$ and $\g_{2'}$ be the loops corresponding to the circles labelled $1'$ and $2'$ in the picture. Define
%$E(B_4): A_1\tensor A_2\ra A_3$ by
%\[ E(B_4)= P(\Si)\comp\m \comp P(\Si_{11'})\tensor P(\Si_{22'}), \]
%where $\Si$ is the bordism in $M$ between $\g_{1'}\star\g_{2'}$ and $\g_3$, $\m$ is the multiplication map determined by $\g_{1'}\star\g_{2'}$, $\Si_{11'}$ is the bordism in $M$ between $\g_{1}$ and $\g_{1'}$ and $\Si_{22'}$ is the bordism in $M$ between $\g_{2}$ and $\g_{2'}$. The definition is independent on the choice of the intermediate 8-like loop.\\ 

\noindent \textbf{$B_5$-type} bordism:
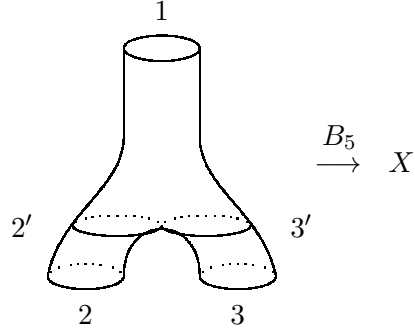
\[ \xy 
0;/r.40pc/:
(0,9)*{1};
(-6,-15)*{2};
(6,-15)*{3};
(-11,-8)*{2'};
(11,-8)*{3'};
(0,3)*\ellipse(3,1){-}; 
(-3,-6)*\ellipse(3,1){.}; 
(-3,-6)*\ellipse(3,1)__,=:a(-180){-}; 
(3,-6)*\ellipse(3,1){.}; 
(3,-6)*\ellipse(3,1)__,=:a(-180){-}; 
(-1.75,-4)*\ellipse(3.5,.8){.}; 
(-1.75,-4)*\ellipse(3.5,.8)__,=:a(-180){-}; 
(1.75,-4)*\ellipse(3.5,.8){.}; 
(1.75,-4)*\ellipse(3.5,.8)__,=:a(-180){-}; 
(-3,-12)*{}="1"; 
(3,-12)*{}="2"; 
(-9,-12)*{}="A2"; 
(9,-12)*{}="B2"; 
"1";"2" **\crv{(-3,-7) & (3,-7)}; 
(-3,6)*{}="A"; 
(3,6)*{}="B"; 
(-3,-1)*{}="A1"; 
(3,-1)*{}="B1"; 
"A";"A1" **\dir{-}; 
"B";"B1" **\dir{-}; 
"B2";"B1" **\crv{(8,-7) & (3,-5)}; 
"A2";"A1" **\crv{(-8,-7) & (-3,-5)}; 
(16,-3)*{ \longrightarrow \ \ X};
(14,-1)*{B_5};
(0,-8)*{}="C";
(-7,-8)*{}="D";
%"D";"C"**\dir{-}
\endxy \]

\vspace{.1in}

\noindent Let $\g_1$ be the top loop and $\g_2$ and $\g_3$ be the bottom two loops. Let $\g_{2'}$ and $\g_{3'}$ be the loops corresponding to the circles labelled $2'$ and $3'$ in the picture. Define
$E(B_5): A_1\ra A_2\tensor A_3$ by
\[ E(B_5)= P(2';2)\tensor P(3';3)\comp \n_{2'3'} \comp P(1;2'3'), \]
where $P(1;2'3')$ is the parallel transport along the path in $LX$ between $\g_1$ and $\g_{2'}\star\g_{3'}$ determined by (the restriction of) $B_4$ by adjunction, $\n_{2'3'}$ is the comultiplication map determined by $\g_{2'}\star\g_{3'}$, $P(2';2)$ is the parallel transport along the path determined by the restriction of $B_4$ joining $\g_{2'}$ and $\g_2$ and $P(3';3)$ denotes the parallel transport along the path between $\g_{3'}$ and $\g_3$. As for $B_4$-type bordisms, the definition is independent on the choice of the intermediate 8-like loop.
The key ingredient we use here is property (7) in the definition of a Frobenius bundle, expressing the compatibility of counits and fission with parallel transport.\\ 

\noindent For an {\it arbitrary} bordism in $X$, i.e. a map $\Si: Z\ra X$ from a compact oriented smooth surface $Z$ (possibly with boundary) into $X$, we decompose $\Si$ into basic bordisms $B_1$-$B_5$ and for each basic bordism we apply the previous construction. The theory that we obtain is certainly compatible under decomposition of bordisms, by the very construction. The only thing left to check is that our theory is {\it topological}. \\

%Invariance under diffeomorphisms.

\noindent {\it Invariance under diffeomorphisms.} The theory $E$ we conctructed is topological if, for any bordism $\Si: Z \ra X$ in $X$ and any diffeomorphism of surfaces $\Phi: Z'\ra Z$ that is the identity on the boundary, we have 
\[ E(\Si\comp\Phi)= E(\Si). \]
This is equivalent to saying that the field theory $E$ is well-defined on $\Si$, i.e. it is independent on how we decompose $\Si$. If $\Si$ is one of the basic bordisms $B_1$-$B_5$, we have already seen that this is the case. For an arbitrary bordism $\Si$, this reduces to 
\begin{itemize}
\item the associativity of fusion and co-assocaitivity of fission expressed by properties (2) and (3) in the definition \ref{FB} of a Frobenius bundle
\item the higher associativity and co-associativity of section 1
\item compatibility of fusion and fission contained in property (4).
\end{itemize}
Let us illustrate this in a couple of examples, which expose the main properties used.\\ 

Consider first a genus two surface $Z$ with one incoming boundary component labelled $a$ and one outgoing boundary component labelled $b$, mapping in two different ways into $X$, as in the picture

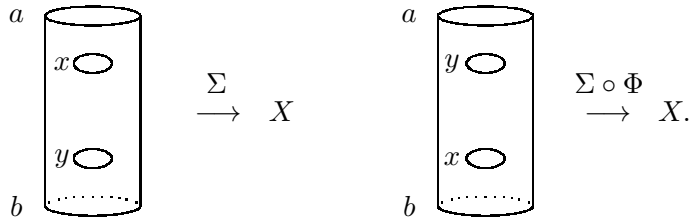
\[ \xy 
0;/r.30pc/:
(-10,5)*\ellipse(5,1){-}; 
(-10,-5)*\ellipse(5,1){.}; 
(-10,-5)*\ellipse(5,1)__,=:a(-180){-}; 
(-10,2.5)*\ellipse(2,1){-}; 
(-10,-2.5)*\ellipse(2,1){-}; 
(-25,10)*{}="TL"; 
(-15,10)*{}="TR"; 
(-25,-10)*{}="BL"; 
(-15,-10)*{}="BR"; 
"TL"; "BL" **\dir{-}; 
"TR"; "BR" **\dir{-}; 
(-4,0)*{ \longrightarrow \ \ X};
(-7,3)*{\Si};
(-28,10)*{a};
(-28,-10)*{b};
(-23.2,5)*{x};
(-23.2,-5)*{y}
\endxy 
\xy 
0;/r.30pc/:
(20,5)*\ellipse(5,1){-}; 
(20,-5)*\ellipse(5,1){.}; 
(20,-5)*\ellipse(5,1)__,=:a(-180){-}; 
(20,2.5)*\ellipse(2,1){-}; 
(20,-2.5)*\ellipse(2,1){-}; 
(35,10)*{}="TL"; 
(45,10)*{}="TR"; 
(35,-10)*{}="BL"; 
(45,-10)*{}="BR"; 
"TL"; "BL" **\dir{-}; 
"TR"; "BR" **\dir{-}; 
(56,0)*{ \longrightarrow \ \ X.};
(53,3)*{\Si\comp \Phi};
(32,10)*{a};
(32,-10)*{b};
(36.5,5)*{y};
(36.5,-5)*{x}
\endxy \]

\vspace{.07cm} 

\noindent Here $\Phi$ is a diffeomorphism of the surface $Z$ that interchanges the holes labelled $x$ and $y$ and is constant on the boundary. Each such bordism can be encoded as two different ways in which the bordism $\Si$ can be decomposed, as indicated by the picture

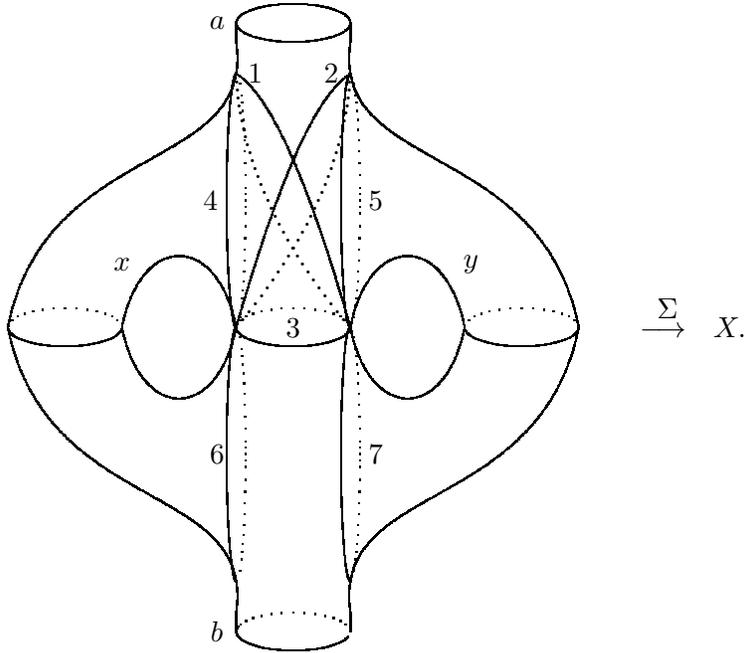
\[ \xy 
0;/r.20pc/:
(0,24)*\ellipse(9,3){-}; 
(-18,0)*\ellipse(9,3){.}; 
(18,0)*\ellipse(9,3){.}; 
(0,0)*\ellipse(9,3){.}; 
(-18,0)*\ellipse(9,3)__,=:a(-180){-}; 
(18,0)*\ellipse(9,3)__,=:a(-180){-}; 
(0,0)*\ellipse(9,3)__,=:a(180){-}; 
(9,0)*{}="1"; 
(27,0)*{}="2"; 
"1";"2" **\crv{(12,15) & (24,15)}; 
(-9,0)*{}="1"; 
(-27,0)*{}="2"; 
"1";"2" **\crv{(-12,15) & (-24,15)}; 
(-45,0)*{}="A2"; 
(45,0)*{}="B2"; 
(-9,48)*{}="A"; 
(9,48)*{}="B"; 
(-9,45)*{}="A1"; 
(9,45)*{}="B1"; 
"A";"A1" **\dir{-}; 
"B";"B1" **\dir{-}; 
"B2";"B1" **\crv{(39,30) & (6,24)}; 
"A2";"A1" **\crv{(-39,30) & (-6,24)}; 
%REFLECT
(0,-24)*\ellipse(9,3){.}; 
(0,-24)*\ellipse(9,3)__,=:a(180){-}; 
(9,0)*{}="1"; 
(27,0)*{}="2"; 
"1";"2" **\crv{(12,-15) & (24,-15)}; 
(-9,0)*{}="1"; 
(-27,0)*{}="2"; 
"1";"2" **\crv{(-12,-15) & (-24,-15)}; 
(-45,0)*{}="A2"; 
(45,0)*{}="B2"; 
(-9,-48)*{}="A"; 
(9,-48)*{}="B"; 
(-9,-45)*{}="A1"; 
(9,-45)*{}="B1"; 
"A";"A1" **\dir{-}; 
"B";"B1" **\dir{-}; 
"B2";"B1" **\crv{(39,-30) & (6,-24)}; 
"A2";"A1" **\crv{(-39,-30) & (-6,-24)}; 
(4.5,10)*\ellipse(1.5,20){.}; 
(4.5,10)*\ellipse(1.5,20)__,=:a(-180){-}; 
(-4.5,10)*\ellipse(1.5,20){.}; 
(-4.5,10)*\ellipse(1.5,20)__,=:a(-180){-}; 
(4.5,-10)*\ellipse(1.5,20){.}; 
(4.5,-10)*\ellipse(1.5,20)__,=:a(-180){-}; 
(-4.5,-10)*\ellipse(1.5,20){.}; 
(-4.5,-10)*\ellipse(1.5,20)__,=:a(-180){-}; 
(-9,0)*{}="1";
(9,40)*{}="2";
"1";"2" **\crv{(-7,20),(1,35)};
"1";"2" **\crv{~*=<4pt>{.}(-7,2),(8,20)};
(9,0)*{}="1";
(-9,40)*{}="2";
"1";"2" **\crv{(0,5),(-1,35)};
"1";"2" **\crv{~*=<4pt>{.}(7,2),(-8,20)};
(0,0)*+{3};
(-6,40)*+{1};
(6,40)*+{2};
(-13,20)*+{4};
(13,20)*+{5};
(-27,10)*+{x};
(28,10)*+{y};
(-12,-20)*+{6};
(13,-20)*+{7};
(-12,48)*+{a};
(-12,-48)*+{b};
(63,0)*{ \longrightarrow \ \ X.};
(59,3)*{\Si};
\endxy \]
The two different ways to travel from $a$ to $b$ along the bordism $\Si$ appear in the following diagram as the upper path respectively lower path.
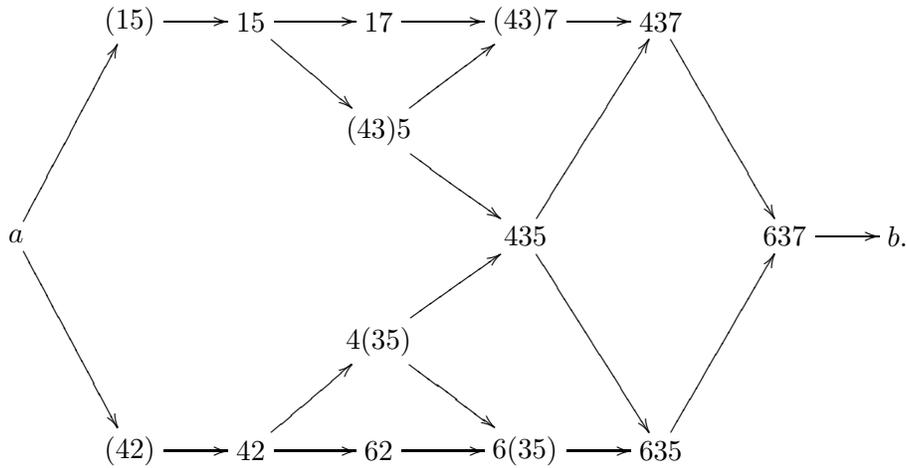
\[ \xymatrix{ & (15) \ar[r] & 15 \ar[r] \ar[dr] & 17 \ar[r] & (43)7 \ar[r] & 437 \ar[ddr] &  & \\
 & & & (43)5 \ar[ur] \ar[dr] &  & & &  \\
a \ar[uur] \ar[ddr] & & & & 435 \ar[uur] \ar[ddr] & & 637 \ar[r] & b.\\
 & & & 4(35) \ar[ur] \ar[dr] &  & & & \\
 & (42) \ar[r] & 42 \ar[r] \ar[ur] & 62 \ar[r] & 6(35) \ar[r] & 635 \ar[uur] &  &  } \]
 All the small n-gons in the diagram are commutative, expressing the compatibility of fusion and fission with parallel transport, or the co-associativity property (3), as in the two ways to reach from $a$ to $435$, or the associativity property (2) appearing in the map $637\ra b$. We conclude that the outer diagram is commutative, the top path giving us the linear map $E(\Si\comp \Phi)$ according to our rules to split a bordism into basic bordisms, and the bottom path giving us $E(\Si)$. \\
 
The second example we would like to consider is of a genus zero surface in $X$ with two incoming boundary components labelled 1 and 2 and two outgoing boundary components labelled 3 and 4.
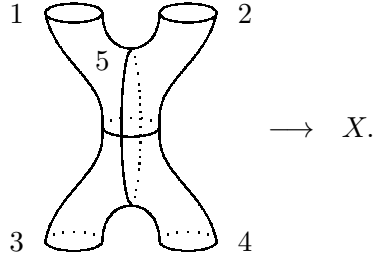
\[ \xy 
0;/r.30pc/:
(-12,12)*{1};
(12,12)*{2};
(-12,-12)*{3};
(12,-12)*{4};
(0,0)*\ellipse(3,1){.}; 
(0,0)*\ellipse(3,1)__,=:a(-180){-}; 
(-3,-6)*\ellipse(3,1){.}; 
(3,-6)*\ellipse(3,1){.}; 
(-3,-6)*\ellipse(3,1)__,=:a(-180){-}; 
(3,-6)*\ellipse(3,1)__,=:a(-180){-}; 
(-3,-12)*{}="1"; 
(3,-12)*{}="2"; 
(-9,-12)*{}="A2"; 
(9,-12)*{}="B2"; 
"1";"2" **\crv{(-3,-7) & (3,-7)}; 
(-3,0)*{}="A"; 
(3,0)*{}="B"; 
(-3,-1)*{}="A1"; 
(3,-1)*{}="B1"; 
"A";"A1" **\dir{-}; 
"B";"B1" **\dir{-}; 
"B2";"B1" **\crv{(8,-7) & (3,-5)}; 
"A2";"A1" **\crv{(-8,-7) & (-3,-5)}; 
%REFLECT 
(-3,6)*\ellipse(3,1){-}; 
(3,6)*\ellipse(3,1){-}; 
(-3,12)*{}="1"; 
(3,12)*{}="2"; 
(-9,12)*{}="A2"; 
(9,12)*{}="B2"; 
"1";"2" **\crv{(-3,7) & (3,7)}; 
(-3,0)*{}="A"; 
(3,0)*{}="B"; 
(-3,1)*{}="A1"; 
(3,1)*{}="B1"; 
"A";"A1" **\dir{-}; 
"B";"B1" **\dir{-}; 
"B2";"B1" **\crv{(8,7) & (3,5)}; 
"A2";"A1" **\crv{(-8,7) & (-3,5)}; 
(0,0)*\ellipse(1,8){.};
(0,0)*\ellipse(1,8)^^,=:a(-180){-}; 
(-3,7)*{5};
(20,0)*{ \longrightarrow \ \ X.};
\endxy \]

\noindent There are basically two ways to descend from top to bottom according to our rules of parallel transport along cylinders and fusion and fission maps. One possible way is to fuse the loops $1$ and $2$, use parallel transport and then disperse into the loops $3$ and $4$. Another possible way is to split loop $1$ into the loops   $3$ and the loop labelled $5$ and then fuse the loops 2 and 5 to reach the loop 4. To show that the result does not depend on the possible ways boils down to considering the following picture in $X$
 
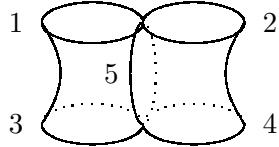
\[ \xy 
0;/r.40pc/:
(-2,2)*\ellipse(4,1.6){-}; 
(2,2)*\ellipse(4,1.6){-}; 
(-2,-2)*\ellipse(4,1.6){.}; 
(2,-2)*\ellipse(4,1.6){.}; 
(-2,-2)*\ellipse(4,1.6)__,=:a(-180){-}; 
(2,-2)*\ellipse(4,1.6)__,=:a(-180){-}; 
(0,0)*\ellipse(1,4){.};
(0,0)*\ellipse(1,4)^^,=:a(-180){-}; 
(-8,4)*{}="1"; 
(8,4)*{}="2"; 
(-8,-4)*{}="3"; 
(8,-4)*{}="4"; 
"1";"3" **\crv{(-6,1) & (-6,-1)}; 
"2";"4" **\crv{(6,1) & (6,-1)}; 
(-10,4)*+{1};
(10,4)*+{2};
(-10,-4)*+{3};
(10,-4)*+{4};
(-2.5,0)*+{5};
\endxy \]
and read off the following diagram
\[ \xymatrix @C=4pc { 1\tensor 2 \ar[r]^{\m_{12}} \ar[d]_P & 12 \ar[d]_P \ar[dr]^P & \\
35 \tensor 2 \ar[r]_-{\m_{(35)2}} \ar[d]_{\n_{35}} & (35)2=3(52) \ar[r]_-P \ar[d]_{\n_{3(52)}} & 34 \ar[d]^{\n_{34}} \\
3\tensor 5\tensor 2 \ar[r]_{\m_{52}} & 3\tensor 52 \ar[r]_P & 3\tensor 4.} \]
All the small n-gons are commutative making the exterior n-gon commutative. The key property we use here is the compatibility of fusion and fission, i.e. property (4) in the definition \ref{FB}.\\

Thus, we have also constructed a topological field theory $E$ over $X$ from a Frobenius bundle with connection over $LX$. The two constructions (2-TFTs $\leadsto$ Frobenius bundles,  Frobenius bundles $\leadsto$ 2-TFTs) are clearly inverses of each other. This concludes our proof of the theorem.

\section{Further remarks}
\subsection{Frobenius actions.} \label{FA}
Let $\g$ be a loop in $X$ based at $x\in X$. We will show that $A_\g$ admits an $A_x$ action and coaction. Let us define first a map 
$ \m: A_x\tensor A_\g\ra A_\g,$
as the composition 
\[ \xymatrix  @C=4pc { A_x\tensor A_\g\ar[r]^-{\m_{x\g}} & A_{x\g} \ar[r]^{P(x\g;\g)} & A_\g, } \]
where $P(x\g;\g)$ denotes the obvious parallel transport between  the loop $x\g$ and the loop $\g$.
To see that this map defines an action, we have to check that the diagram
\[ \xymatrix{ A_x\tensor A_x \tensor A_\g \ar[rr]^{1\tensor \m} \ar[d]_{\m_x\tensor 1}& &A_x\tensor A_\g \ar[d]^{1\tensor\m} \\
A_x\tensor A_\g \ar[rr]^\m & & A_\g} \]
%and the diagram 
%\[ \xymatrix{ A_x \tensor A_\g \ar[rr]^{\n_x\tensor 1} \ar[d]_{\m}& &A_x\tensor A_x \tensor A_\g \ar[d]^{1\tensor \m} \\
%A_\g & & A_x\tensor A_\g \ar[ll]^\m } \]
is commutative. It suffices to notice that in the diagram below all but the front face are commutative
\[ \xymatrix @C=4pc { & A_x\tensor A_{x\g} \ar[dr]^{1\tensor P(x\g;\g)} \ar'[d][dd]^{\m_{xx\g}}& \\
A_x\tensor A_x \tensor A_\g \ar[ur]^{1\tensor \m_{x\g}} \ar[rr]^{1\tensor \m\ \ \ \ \ \ \ \ } \ar[dd]_{\m_x\tensor 1}& &A_x\tensor A_\g \ar[dd]^{\m_{x\g}} \\
& A_{xx\g} \ar[dr]^{P(xx\g;x\g)} & \\
A_x\tensor A_\g \ar[ur]^{\m_{xx\g}} \ar[rr]^{\m_{x\g}} & & A_{x\g}} \]
making the front face commute as well. Similarly, we define a coaction map $\n: A_\g\ra A_x\tensor A_\g$ as the composition
\[ \xymatrix  @C=4pc { A_\g\ar[r]^{P(\g;x\g)} & A_{x\g} \ar[r]^{\n_{x\g}} & A_x\tensor A_\g, } \]
where $P(\g;x\g)$ denotes the obvious parallel transport between  the loop $\g$ and the loop $x\g$.
To see that this map indeed defines a co-action, we have to check that the diagram
\[ \xymatrix{A_\g\ar[rr]^\n \ar[d]_\n & &A_x\tensor A_\g \ar[d]^{1\tensor\n} \\
A_x\tensor A_\g \ar[rr]^-{\n_x\tensor 1} & & A_x\tensor A_x \tensor A_\g} \]
commutes. This is done as before. Next, we will check that the two actions are compatible in the sense that the following diagram commutes
\[ \xymatrix @C=4pc { A_x\tensor A_\g \ar[r]^-\m \ar[d]_{1\tensor \n} & A_\g \ar[d]^\n \\
A_x\tensor A_x\tensor A_\g \ar[r]_{\m_x\tensor 1} & A_x\tensor A_\g. }\]
First, let us notice that
\begin{eqnarray*}
\n\m &=& \n_{x\g}\comp P(\g;x\g) \comp P(x\g;\g)\comp \m_{x\g}\\
&=& \n_{x\g}\comp\m_{x\g}.
\end{eqnarray*}
Then, consider the following diagram
\[ \xymatrix @C=3pc { A_x\tensor A_\g \ar[rrr]^{\m_{x\g}} \ar[ddd]_{1\tensor\n} \ar[dr]_{1\tensor P(\g;x\g)} & & & A_{x\g} \ar[ddd]^{\n_{x\g}} \\
& A_x\tensor A_{x\g} \ar[r]^{\m_{xx\g}} \ar[ddl]_{1\tensor\n_{x\g}} & A_{xx\g} \ar[ur]_{P(xx\g;x\g)} \ar[d]^{\n_{xx\g}} & \\
& & A_{xx}\tensor A_\g \ar[dr]_{P(xx;x)\tensor 1} & \\
A_x\tensor A_x\tensor A_\g \ar[rrr]_{\m_x \tensor 1} \ar[urr]^{\m_x \tensor 1} & & & A_x\tensor A_\g. } \]
All the inside n-gons are commutative (note that $A_{xx}=A_x$ and $P(xx;x)$ is the identity on $A_x$), and therefore the outer square is commutative. This proves the compatibility of module and comodule structures. We conclude that the fibers admit a Frobenius action of the Frobenius algebra over their basepoints, justifying hopefully our terminology of ``Frobenius bundles''.

\begin{lem} (Frobenius actions are compatible with reparametrizations of loops.) Let $\g:S^1\ra X$ be a loop in $X$ based at $x\in X$, and let the loop $\ti{\g}=\g\comp \f$ based at $y\in X$, for some diffeomorphism $\f$ of $S^1$ preserving its orientation. The following diagrams commute
\[ \xymatrix  @C=3pc{ A_x\tensor A_\g \ar[r]^-\m \ar[d]_{P(x;y)\tensor P(\g;\ti{\g})} & A_\g \ar[d]^{P(\g;\ti{\g})} & &    A_\g \ar[r]^-\n \ar[d]_{P(\g;\ti{\g})} &   A_x\tensor A_\g \ar[d]^{P(x;y)\tensor P(\g;\ti{\g})}\\
A_y\tensor A_{\ti{\g}} \ar[r]_-\m &  A_{\ti{\g}} & &  A_{\ti{\g}} \ar[r]_-\n & A_y\tensor A_{\ti{\g}}} \]
where $P(\g;\ti{\g})$ is the parallel transport along the {\it canonical} path $\G$ between $\g$ and $\ti{\g}$ and $P(x;y)$ is the parallel transport along the paths of constant loops between the loop at $x$ and the loop at $y$, determined by the path $\G$.
\end{lem} 

\begin{proof} The proof consists in following the definitions. Let us verify that the first diagram commutes (the second diagram is dealt with in a similar fashion). Indeed, consider the diagram
\[ \xymatrix  @C=4pc{ A_x\tensor A_\g \ar[dr]_{\m_{x\g}} \ar[rr]^-\m \ar[ddd]_{P(x;y)\tensor P(\g;\ti{\g})} & & A_\g \ar[ddd]^{P(\g;\ti{\g})} \\
& A_{x\g} \ar[ur]_{P(x\g;\g)} \ar[d]_{P(x\g;y\ti{\g})} & \\
& A_{y\ti{\g}} \ar[dr]^{P(y\ti{\g}; \ti{\g})} & \\
A_y\tensor A_{\ti{\g}} \ar[ur]^{\m_{y\ti{\g}}} \ar[rr]_-\m & &  A_{\ti{\g}}. } \]
All the inside diagrams are commutative, making the outer diagram commutative.
\end{proof}

\subsection{On holonomy}
Start with a Frobenius bundle $A$ with connection over $LX$. Let $Z$ be a closed surface and $\Si:Z\ra X$ be a closed ``surface'' in $X$. Let $a$ and $b$ two distinct points on $Z$ and $x$, respectively $y$, their images via $\Si$. The {\it holonomy} around the closed surface $\Si$ is defined to be the composition
\[ \xymatrix @C=4pc { k \ar[r]^{\eta_x} & A_x \ar[r]^{E(\ti{\Si})}  & A_y \ar[r]^{\e_y}  & k,} \]
where $A_x$, $A_y$ denote the constant loops at the points $x$ and $y$, and $\ti{\Si}$ denotes the surface $\Si$ in $X$  viewed as a bordism between the constant loop at $x$ to the constant loop at $y$. To see that holonomy is well-defined, i.e. it does not depend on our choices, let us remark that it is enough to consider the case of a genus zero surface with one boundary component mapping into $X$  and show that the transport along such surfaces is well-defined. These cases are exactly covered by the $B_2$ and $B_3$-type bordisms in $X$ that appeared in the proof of the theorem, for which we showed independence on the various choices. This makes holonomy well-defined.

\subsection{$\dif(S^1)^+$-action on a Frobenius bundle} Let $A$ be a Frobenius bundle with connection over $LX$. We define an $\dif(S^1)^+$-action on the bundle $A$ as follows. Let $\g$ be a loop in $X$ and $\f$ an element of $\dif(S^1)^+$. Define a map
\[ R_\f: A_\g \ra A_{\g\comp \f}, \ \ \ R_\f:=P(\g;\g\comp\f), \]
where $P(\g;\g\comp\f)$ is the parallel transport along  the {\it canonical} path  $\G:I\ra LX$ in the loopspace $LX$ between $\g$ and $\g\comp\f$ given by 
\[ \G_t= \g\comp\big ((1-t)id_{S^1}+t\f \big ). \]
Here we identify a diffeomorphism of $S^1$ with a diffeomorphism of $I=[0,1]$ preserving the endpoints. If we let the loop $\g$ vary in $LX$, for each diffeomorphism $\f$ of $S^1$, we get a bundle map (still denoted) 
$R_\f:A\ra A$.

\begin{lem}
The maps $R_\f, \ \f\in\dif(S^1)^+$ define a right $\dif(S^1)^+$-action on the Frobenius bundle $A$ over $LX$. 
In other words,
\[ R_\psi\comp R_\f= R_{\f\comp\psi}, \]
for $\f, \psi$ in $\dif(S^1)^+$.
\end{lem}

\begin{proof}
We have to check that for any loop $\g$ in $X$ and any diffeomorphisms $\f, \psi$ in $\dif(S^1)^+$, we have
\[ P(\g;\g\f\psi)= P(\g\f;\g\f\psi)\comp P(\g;\g\f).\]
If $\psi\geq\f\geq id_{S^1}$, then this is indeed the case, since the path from $id_{S^1}$ to $\f\psi$ passes through the intermediate step $\f$. Otherwise we encounter a situation like in the picture below

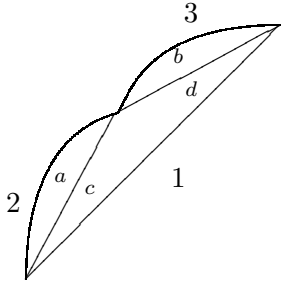
\[ \xy
0;/r.40pc/:
(0,0);(20,20)**\dir{-};
(10,10)+(-3,3)="o";
(0,0);"o"**\dir{-};
"o";(20,20)**\dir{-};
(0,0);"o" **\crv{(0,2) & (0,11)}; 
"o";(20,20) **\crv{(8,13) & (8,20)}; 
(12,8)*{1};
(-1,6)*{2};
(13,21)*{3};
(2.7,8)*{\scriptstyle a};
(5,7)*{\scriptstyle c};
(13,15)*{\scriptstyle d};
(12,17.6)*{\scriptstyle b};
\endxy \]
where 1 labels the identity on $S^1$, 2 labels the diffeomorphism $\f$ and 3 denotes the composition $\f\psi$. Then 2 is the concatenation of the paths labelled $a$ and $d$, and 3 is the concatenation of the paths labelled $c$ and $b$. We would like to show that 
\[ P(1;3)=P(2;3)\comp P(1;2), \]
where $P(i;j):=P(\g\comp i; \g\comp j)$. We have 
\begin{eqnarray*}
P(2;3)\comp P(1;2) &= & P(cd;ad)\comp P(cb;cd)\comp P(cd;cb)\comp P(1;cd)\\ 
& = & P(cd;ad)\comp P(1;cd) \\
&=& P(1;3).
\end{eqnarray*}
All the other displacements of the diffeomorphisms 1, 2 and 3 are obtained through a repetitive process of the situations described above. This proves the lemma.

\end{proof}

\subsection{Rank-one 2-dimensional topological field theories} 
A rank-one (i.e. the fibers are one-dimensional) TFT over $X$  gives rise to a fusion bundle with superficial connection over $LX$ in the sense of Waldorf. Indeed, we only need to check a couple of things. First, is to give a fusion map that is strictly associative. To see this, consider the paths (constant at the endpoints) labelled 1, 2, and 3 between two points $x$ and $y$ in $X$ as in the picture
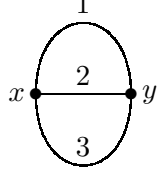
\[ \xy
0;/r.15pc/:
(0,19)*{1};
(0,4)*{2};
(0,-11)*{3};
(-10,0)*{}="x";
(10,0)*{}="y";
"x"+(-4,0) *{x};
"y"+(4,0) *{y};
"x"*+{\bullet};
"y"*+{\bullet};
"x"; "y" **\dir{-}; 
"x";"y" **\crv{(-10,20) & (10,20)}; 
"x";"y" **\crv{(-10,-20) & (10,-20)}; 
\endxy \]
Each such picture should provide a {\it fusion} map (in the sense of Waldorf) 
\[ \l_{123}:A_{1\bar{2}}\tensor A_{2\bar{3}} \ra A_{1\bar{3}}, \]
where the bar notation is used for traveling along a path backwards. We define this to be the composition
\[ \l_{123}:= P(1\bar{2}2\bar{3};1\bar{3})\comp\m_{(1\bar{2})(2\bar{3})}. \]
We would have to check the following associativity (refer to the picture below)
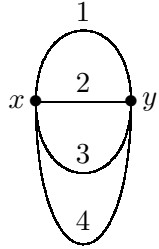
\[ \xy
0;/r.15pc/:
(0,19)*{1};
(0,4)*{2};
(0,-11)*{3};
(0,-25)*{4};
(-10,0)*{}="x";
(10,0)*{}="y";
"x"+(-4,0) *{x};
"y"+(4,0) *{y};
"x"*+{\bullet};
"y"*+{\bullet};
"x"; "y" **\dir{-}; 
"x";"y" **\crv{(-10,20) & (10,20)}; 
"x";"y" **\crv{(-10,-20) & (10,-20)}; 
"x";"y" **\crv{(-10,-40) & (10,-40)}; 
\endxy \]

\[ \xymatrix @C=4pc { 1\bar{2} \tensor 2\bar{3} \tensor 3\bar{4} \ar[r]^{\l_{234}} \ar[d]_{\l_{123}} & 1\bar{2}\tensor 2\bar{4} \ar[d]^{\l_{124}} \\
1\bar{3}\tensor 3\bar{4} \ar[r]_{\l_{134}} & 1\bar{4}, } \]
where we dropped the $A$'s when denoting the fibers to simplify notation. Indeed we have
\begin{eqnarray*}
\l_{134}\comp \l_{123} & = & P(1\bar{3}3\bar{4};1\bar{4})\comp \m_{(1\bar{3})(3\bar{4})}\comp P(1\bar{2}2\bar{3};1\bar{3})\comp\m_{(1\bar{2})(2\bar{3})} \\
&= & P(1\bar{3}3\bar{4};1\bar{4})\comp P(1\bar{2}2\bar{3}3\bar{4};1\bar{3}3\bar{4})\comp \m_{(1\bar{2}2\bar{3})3\bar{4}} \comp\m_{(1\bar{2})(2\bar{3})} \\
&= & P(1\bar{2}2\bar{3}3\bar{4};1\bar{4})\comp \m_{(1\bar{2})(2\bar{3})(3\bar{4}).} 
\end{eqnarray*}
The second equality is true since fusion is compatible with parallel transport, and the third equality expresses the associativity of fusion. On the other side, we have
\begin{eqnarray*}
\l_{124}\comp \l_{234} & = & P(1\bar{2}2\bar{4};1\bar{4})\comp \m_{(1\bar{2})(2\bar{4})}\comp P(2\bar{3}3\bar{4};2\bar{4})\comp\m_{(2\bar{3})(3\bar{4})} \\
&= & P(1\bar{2}2\bar{4};1\bar{4})\comp P(1\bar{2}2\bar{3}3\bar{4};1\bar{2}2\bar{4})\comp \m_{(1\bar{2})(2\bar{3}3\bar{4})} \comp\m_{(2\bar{3})(3\bar{4})} \\
&= & P(1\bar{2}2\bar{3}3\bar{4};1\bar{4})\comp \m_{(1\bar{2})(2\bar{3})(3\bar{4}).} 
\end{eqnarray*}

The second thing to check is that out of a connection on a Frobenius bundle over $LX$ we obtain a superficial connection, i.e. a notion of parallel transport $P$ along paths in $LX$  so that any two paths in $LX$ that are rank-two-homotopic give rise to the same parallel transport (two paths $\G$ and $\G'$ are {\it rank-two homotopic} if there is a homotopy $H:I\times I \ra LX$ connecting $\G$ and $\G'$ so that the adjoint map $\check{H}:I\times I\times S^1\ra X$ has rank two, cf. \cite{Wal1}). If $\G'$ is obtained from $\G$ by a precomposition with a diffeomorphism of the cylinder, then $P(\G')=P(\G)$, by the definition of a connection. If this is not the case, they must be related by precomposition with a surjective smooth map $\Phi:I\times S^1\ra I\times S^1$. In this case, the claim is that $P(\G')=E(\check{\G}\comp\Phi)$ can be written as a composition of maps $E(\check{\G}\comp\Phi_i)$, for some diffeomorphisms $\Phi_i$, interspersed with parallel transport maps along paths that retrace loops, along with their inverses. Therefore, diffeomorphic-invariant parallel transport implies rank-two homotopic parallel transport. The argument sketched here is probably more transparent in the one-dimesional case of parallel transport along usual paths in a space.

Now, Waldorf \cite{Wal1} shows that a fusion bundle with superficial connection gives rise to an $S^1$-bundle gerbe with connection over $X$. Conversely, an $S^1$-bundle gerbe with connection over $X$ gives rise to a rank-one 2-TFT over $X$ by the work of Gawedzki, see \cite{Ga} and \cite{GR}. We can conclude that there is an equivalence between rank-one 2-TFTs over $X$  and $S^1$-bundle gerbes with connections over $X$.

\bibliographystyle{plain}
\bibliography{bibliografie}

\bigskip
\raggedright Max-Planck-Institut  f\"{u}r Mathematik\\  53111 Bonn, Germany\\
Email: {\tt florin@mpim-bonn.mpg.de}

\end{document}